\documentclass[12pt]{amsart}
\usepackage{amsfonts,amssymb,amscd,amsmath,enumerate,verbatim}
\usepackage[latin1]{inputenc}
\usepackage{amscd}
\usepackage{latexsym}
\usepackage{mathptmx}
\usepackage{multicol}
\input xy
\xyoption{all}

\def\NZQ{\mathbb}               

\def\ZZ{{\NZQ Z}}

\def\frk{\mathfrak}               

\def\Phi{{\frk N}}
\def\opn#1#2{\def#1{\operatorname{#2}}} 
\opn\chara{char} 
\opn\length{\ell} 
\opn\pd{pd} 
\opn\rk{rk}
\opn\projdim{proj\,dim} 
\opn\injdim{inj\,dim} 
\opn\rank{rank}
\opn\depth{depth} 
\opn\grade{grade} 
\opn\height{height}
\opn\embdim{emb\,dim} 
\opn\codim{codim}

\opn\Tr{Tr} 
\opn\bigrank{big\,rank}
\opn\superheight{superheight}
\opn\lcm{lcm}
\opn\trdeg{tr\,deg}
\opn\reg{reg} 
\opn\lreg{lreg} 
\opn\ini{in} 
\opn\lpd{lpd}
\opn\size{size}
\opn\mult{mult}
\opn\dist{dist}
\opn\cone{cone}
\opn\lex{lex}
\opn\rev{rev}
\opn\div{div} \opn\Div{Div} \opn\cl{cl} \opn\Cl{Cl}
\opn\Spec{Spec} \opn\Supp{Supp} \opn\supp{supp} \opn\Sing{Sing}
\opn\Ass{Ass} \opn\Min{Min}
\opn\Ann{Ann} \opn\Rad{Rad} \opn\Soc{Soc}
\opn\Syz{Syz} \opn\Im{Im} \opn\Ker{Ker} \opn\Coker{Coker}
\opn\Am{Am} \opn\Hom{Hom} \opn\Tor{Tor} \opn\Ext{Ext}
\opn\End{End} \opn\Aut{Aut} \opn\id{id} \opn\ini{in}

\opn\nat{nat}
\opn\pff{pf}
\opn\Pf{Pf} \opn\GL{GL} \opn\SL{SL} \opn\mod{mod} \opn\ord{ord}
\opn\Gin{Gin}
\opn\Hilb{Hilb}\opn\adeg{adeg}\opn\std{std}\opn\ip{infpt}
\opn\Pol{Pol}
\opn\sat{sat}
\opn\Var{Var}
\opn\Gen{Gen}

\opn\aff{aff} \opn\con{conv} \opn\relint{relint} \opn\st{st}
\opn\lk{lk} \opn\cn{cn} \opn\core{core} \opn\vol{vol}
\opn\link{link} \opn\star{star}
\opn\gr{gr}

\def\pot#1#2{#1[\kern-0.28ex[#2]\kern-0.28ex]}

\opn\dirlim{\underrightarrow{\lim}}
\opn\inivlim{\underleftarrow{\lim}}
\let\to=\rightarrow

\def\Implies{\ifmmode\Longrightarrow \else
        \unskip${}\Longrightarrow{}$\ignorespaces\fi}
\def\implies{\ifmmode\Rightarrow \else
        \unskip${}\Rightarrow{}$\ignorespaces\fi}
\def\iff{\ifmmode\Longleftrightarrow \else
        \unskip${}\Longleftrightarrow{}$\ignorespaces\fi}

\let\:=\colon
\newtheorem{Theorem}{Theorem}[section]
\newtheorem{Lemma}[Theorem]{Lemma}

\newtheorem{Proposition}[Theorem]{Proposition}
\newtheorem{Remark}[Theorem]{Remark}

\newtheorem{Example}[Theorem]{Example}

\newtheorem{Conjecture}[Theorem]{Conjecture}

\let\epsilon\varepsilon
\let\phi=\varphi
\let\kappa=\varkappa
\def\qed{\ifhmode\textqed\fi
      \ifmmode\ifinner\quad\qedsymbol\else\dispqed\fi\fi}
\def\textqed{\unskip\nobreak\penalty50
       \hskip2em\hbox{}\nobreak\hfil\qedsymbol
       \parfillskip=0pt \finalhyphendemerits=0}
\def\dispqed{\rlap{\qquad\qedsymbol}}

\opn\dis{dis}
\opn\height{height}
\opn\dist{dist}
\def\pnt{{\raise0.5mm\hbox{\large\bf.}}}

\opn\Lex{Lex}

\begin{document}
\title{Regularity and h-polynomials of monomial ideals}
\author{Takayuki Hibi and Kazunori Matsuda}
\address{Takayuki Hibi,
Department of Pure and Applied Mathematics,
Graduate School of Information Science and Technology,
Osaka University, Suita, Osaka 565-0871, Japan}
\email{hibi@math.sci.osaka-u.ac.jp}
\address{Kazunori Matsuda,
Department of Pure and Applied Mathematics,
Graduate School of Information Science and Technology,
Osaka University, Suita, Osaka 565-0871, Japan}
\email{kaz-matsuda@ist.osaka-u.ac.jp}
\subjclass[2010]{05E40, 13H10}
\keywords{Castelnuovo--Mumford regularity, $h$-polynomial, Cameron--Walker graph}
\begin{abstract}
Let $S = K[x_1, \ldots, x_n]$ denote the polynomial ring in $n$ variables over a field $K$ with each $\deg x_i = 1$ and $I \subset S$ a homogeneous ideal of $S$ with $\dim S/I = d$.  The Hilbert series of $S/I$ is of the form $h_{S/I}(\lambda)/(1 - \lambda)^d$, where $h_{S/I}(\lambda) = h_0 + h_1\lambda + h_2\lambda^2 + \cdots + h_s\lambda^s$ with $h_s \neq 0$ is the $h$-polynomial of $S/I$.  It is known that, when $S/I$ is Cohen--Macaulay, one has $\reg(S/I) = \deg h_{S/I}(\lambda)$, where $\reg(S/I)$ is the (Castelnuovo--Mumford) regularity of $S/I$.  In the present paper, given arbitrary integers $r$ and $s$ with $r \geq 1$ and $s \geq 1$, a monomial ideal $I$ of $S = K[x_1, \ldots, x_n]$ with $n \gg 0$ for which $\reg(S/I) = r$ and $\deg h_{S/I}(\lambda) = s$ will be constructed.  Furthermore, we give a class of edge ideals $I \subset S$ of Cameron--Walker graphs with $\reg(S/I) = \deg h_{S/I}(\lambda)$ for which $S/I$ is not Cohen--Macaulay.   
\end{abstract}
\maketitle
\section*{Introduction}
Let $S = K[x_1, \ldots, x_n]$ denote the polynomial ring in $n$ variables over a field $K$ with each $\deg x_i = 1$ and $I \subset S$ a homogeneous ideal of $S$ with $\dim S/I = d$.  The Hilbert series $H_{S/I}(\lambda)$ of $S/I$ is of the form $H_{S/I}(\lambda) = (h_0 + h_1\lambda + h_2\lambda^2 + \cdots + h_s\lambda^s)/(1 - \lambda)^d$, where each $h_i \in \ZZ$ (\cite[Proposition 4.4.1]{BH}).  We say that $h_{S/I}(\lambda) = h_0 + h_1\lambda + h_2\lambda^2 + \cdots + h_s\lambda^s$ with $h_s \neq 0$ is the {\em $h$-polynomial} of $S/I$.  Let $\reg(S/I)$ denote the ({\em Castelnuovo--Mumford}\,) {\em regularity} \cite[p.~168]{BH} of $S/I$.  A well-known fact (e.g., \cite[Lemma 2.5]{BV}) is that, when $S/I$ is Cohen--Macaulay, one has
\[
\reg(S/I) = \deg h_{S/I}(\lambda).
\] 
Its converse is false.  The following example was found by J\"urgen Herzog.  Let $I \subset S = K[x_1,x_2,x_3,x_4]$ be the monomial ideal $(x_2^3, x_2^2x_3, x_2x_3^2, x_3^3, x_1^2, x_1x_2, x_1x_3, x_1x_4)$, which is strongly stable (\cite[p.~103]{HH}).  Then $\dim S/I = 1$, $\depth S/I = 0$, $\reg(S/I) = 2$ and $h_{S/I}(\lambda) = 1 + 3\lambda + 2\lambda^2$.  At this stage, it is reasonable to discover a natural class of monomial ideals $I \subset S = K[x_1, \ldots, x_n]$ for which $S/I$ is not Cohen--Macaulay with $\reg(S/I) = \deg h_{S/I}(\lambda)$.  

Furthermore, one cannot escape from the temptation to present the following 

\begin{Conjecture}
\label{conj}
Given arbitrary integers $r$ and $s$ with $r \geq 1$ and $s \geq 1$, there exists a strongly stable ideal $I$ of $S = K[x_1, \ldots, x_n]$ with $n \gg 0$ for which $\reg(S/I) = r$ and $\deg h_{S/I}(\lambda) = s$.  
\end{Conjecture}

It follows from \cite[Lemma 4.1.3]{BH} that if $I$ has a {\em pure resolution} (\cite[p.~153]{BH}), then
\[
\deg h_{S/I}(\lambda) - \reg(S/I) = \dim S/I - \depth S/I. 
\] 
As a result, when $1 \leq r \leq s$, a desired ideal can be found in the class of squarefree lexsegment ideals (\cite{AHH}, \cite[p.~124]{HH}).  

The purpose of the present paper is to give an affirmative answer to a weak version of Conjecture \ref{conj}, that is to say, a monomial ideal $I$ of $S = K[x_1, \ldots, x_n]$ for which $\reg(S/I) = r$ and $\deg h_{S/I}(\lambda) = s$ will be constructed. 

\begin{Theorem}
\label{main}
Given arbitrary integers $r$ and $s$ with $r \geq 1$ and $s \geq 1$, there exists a monomial ideal $I$ of $S = K[x_1, \ldots, x_n]$ with $n \gg 0$ for which $\reg(S/I) = r$ and $\deg h_{S/I}(\lambda) = s$.  
\end{Theorem} 

When $r > s$, a basic process in order to obtain a required ideal of Theorem \ref{main} is to find a monomial ideal $I = I_N \subset S$ with $\reg(S/I) = N + 1$ and $\deg h_{S/I}(\lambda) = 1$ for an arbitrary integer $N > 0$.  A proof of Theorem \ref{main} will be achieved in Section $1$. 

On the other hand, in Section $2$, we give a class of edge ideals $I \subset S$ of Cameron--Walker graphs (\cite{HHKO}) with $\reg(S/I) = \deg h_{S/I}(\lambda)$ for which $S/I$ is not Cohen--Macaulay.

\section{Proof of Theorem \ref{main}}
Before giving a proof of Theorem \ref{main}, several lemmata will be prepared.
Let, as before, $S = K[x_1, \ldots, x_n]$ denote the polynomial ring in $n$ variables over a field $K$ with each $\deg x_i = 1$.  Lemma \ref{LemA} below follows immediately from the definition of regularity in terms of graded Betti numbers (\cite[p.~48]{HH}).  

\begin{Lemma}\label{LemA}
Let $I \subset S$ be a proper homogeneous ideal.  Then $\reg (S/I) = \reg (I) - 1$.  
\end{Lemma}

\begin{Lemma}[{\cite[Lemma 3.2]{HT}}]
\label{LemB}
Let $S_{1} = K[x_{1}, \ldots, x_{m}]$ and $S_{2} = K[y_{1}, \ldots, y_{n}]$ be polynomial rings over a field $K$.  Let $I_{1}$ be a nonzero homogeneous ideal of $S_{1}$ and $I_2$ that of $S_2$.  Write $S$ for $S_{1} \otimes_{K} S_{2} = K[x_{1}, \ldots, x_{m}, y_{1}, \ldots, y_{n}]$ and regard $I_{1} + I_{2}$ and $I_{1}I_{2}$ as homogeneous ideals of $S$. Then
\begin{enumerate}
	\item[$(1)$] $\reg (I_{1}I_{2}) = \reg (I_{1}) + \reg (I_{2})$\,$;$ 
	\item[$(2)$] $\reg (I_{1} + I_{2}) = \reg (I_{1}) + \reg (I_{2}) - 1$\,$;$
	\item[$(3)$] $\reg (S/(I_{1} + I_{2})) = \reg (S_{1}/I_{1}) + \reg (S_{2}/I_{2})$. 
\end{enumerate}
\end{Lemma}

\begin{Lemma}[{\cite[Lemma 2.10]{DHS}}]
\label{LemC}
Let $I \subset S$ be a monomial ideal and $x_i$ a variable of $S$ which appears in a monomial belonging to the unique minimal system of monomial generators of $I$.  Then    
\[
\reg (I) \le \max\{ \, \reg (I : (x)) + 1, \, \reg (I + (x)) \, \}. 
\]
\end{Lemma}

In the first step of a proof of Theorem \ref{main}, given integers $1 \le r \le s$, we construct a monomial ideal $I \subset S$ with $\reg (S/I) = r$ and $h_{S/I} (\lambda) = s$.  As was mensioned in Introduction, a desired ideal $I$ can be found in the class of squarefree lexsegment ideals. 

Let $<_{\mathrm{lex}}$ denote the lexicographic order (\cite[p.~24]{HH}) on $S = K[x_{1}, \ldots, x_{n}]$ induced from $x_{1} > x_{2} > \cdots > x_{n}$.  A monomial ideal $I \subset S$ is called {\em squarefree lexsegment} if $I$ is generated by squarefree monomials and if, for all 
squarefree monomials $u \in I$ and for all squarefree monomials $v \in S$ with 
$\deg u = \deg v$ and $u <_{\mathrm{lex}} v$, one has $v \in I$.

Fix 
integers $r$ and $s$ with $1 \le r \le s$ and consider the squarefree lexsegment ideal 
\[
I _{r, s} = (u_{1}u_{2} \cdots u_{r}u_{r+1}, u_{1}u_{2} \cdots u_{r}u_{r+2}, \ldots, u_{1}u_{2} \cdots u_{r}u_{s+1})
\]
of the polynomial ring $K[u_{1}, \ldots, u_{s+1}]$ in $(s + 1)$ variables over a field $K$.  

\begin{Proposition}\label{PropA}
One has
\begin{enumerate}
	\item[$(1)$] $\reg (K[u_{1}, \ldots, u_{s+1}] / I_{r, s}) = r$\,$;$ 	
	\item[$(2)$] $\displaystyle H_{K[u_{1}, \ldots, u_{s+1}] / I_{r, s}} (\lambda) = \frac{1 + \lambda + \cdots + \lambda^{r-1} + \lambda^{r}(1-\lambda)^{s-r}}{(1-\lambda)^{s}}$.\\
\end{enumerate}	
Thus, in particular, $\deg h_{K[u_{1}, \ldots, u_{s+1}] / I_{r, s}} (\lambda) = s$. 
\end{Proposition}
\begin{proof}
(1) Since $I_{r, s} = (u_{1}u_{2} \cdots u_{r}) (u_{r+1}, \ldots, u_{s+1})$, it follows from Lemma \ref{LemB} (1) that $\reg (I_{r, s}) = r + 1$. 
Hence Lemma \ref{LemA} says that $\reg (K[u_{1}, \ldots, u_{s+1}] / I_{r, s}) = r$, as desired.

\smallskip

(2) Let $u = u_{1}u_{2} \cdots u_{r}$.  Then $I_{r, s} + (u) = (u)$ and $I_{r, s} : (u) = (u_{r+1}, \ldots, u_{s+1})$.  Thus the short exact sequence
\[
0 \to \frac{K[u_{1}, \ldots, u_{s+1}]}{I_{r, s} : (u)} (-r) \xrightarrow{\ \times u \ } \frac{K[u_{1}, \ldots, u_{s+1}]}{I_{r, s}} \to \frac{K[u_{1}, \ldots, u_{s+1}]}{I_{r, s} + (u)} \to 0, 
\]
yields
\begin{eqnarray*}
H_{K[u_{1}, \ldots, u_{s+1}] / I_{r, s}} (\lambda) &=&H_{K[u_{1}, \ldots, u_{s+1}] / (u)} (\lambda) + \lambda^{r} \cdot H_{K[u_{1}, \ldots, u_{s+1}] / (u_{r+1}, \ldots, u_{s+1})} (\lambda) \\
& & \\
&=& \frac{1 + \lambda + \cdots + \lambda^{r-1}}{(1 - \lambda)^{s}} + \frac{\lambda^{r}}{(1-\lambda)^{r}} \\
& & \\
&=& \frac{1 + \lambda + \cdots + \lambda^{r-1} + \lambda^{r}(1-\lambda)^{s-r}}{(1 - \lambda)^{s}}, 
\end{eqnarray*}
as required.
\qed

\end{proof}

Proposition \ref{PropA} guarantees that, when $1 \leq r \leq s$, Conjecture \ref{conj} is true.  

Now, in the second step of a proof of Theorem \ref{main}, we turn to the discussion of finding a desired monomial ideal for $1 \leq s < r$.  

Let $n \geq 2$ and 
\[
S_{n} = K[x, y_{1}, \ldots, y_{n}, z_{1}, \ldots, z_{n + 1}]
\]   
the polynomial ring in $2(n + 1)$ variables over a field $K$.  We then introduce the monomial ideals $I_{n}, J_{n}, K_{n}$ and $L_{n}$ defined as follows: 

\newpage
 

\begin{eqnarray*}
I_{n} &=& (xy_{1}y_{2} \cdots y_{n}, xz_{1}z_{2} \cdots z_{n+1}) 
+ \sum_{i = 1}^{n - 2} (y_{1}, \ldots, y_{i}, y_{i+1} \cdots y_{n})(z_{i}) \\
& & \, \, \, \, \, \, \, \, \, \, \, \, \, \, \, \, \, \, \, \, + \, (y_{1}, \ldots, y_{n})(z_{n - 1}, z_{n}, z_{n + 1}), \\
& & \\
J_{n} &=& (xz_{1}z_{2} \cdots z_{n+1}) + \sum_{i = 1}^{n - 1}(z_{i}, \ldots, z_{n + 1})(y_{i}), \\
& & \\
K_{n} &=& (xy_{1}y_{2} \cdots y_{n-1}) + \sum_{i = 1}^{n - 2} (y_{1}, \ldots, y_{i}, y_{i+1} \cdots y_{n-1})(z_{i}), \\
& & \\
L_{n} &=& \sum_{i = 1}^{n - 1} (z_{i}, \ldots, z_{n})(y_{i}). 
\end{eqnarray*}

\medskip

\begin{Remark}\label{IJKL}
One has
\begin{enumerate}
	\item[$(1)$] $I_{n} + (y_{n}) = (y_{n}) + J_{n}$\,$;$ 
	\item[$(2)$] $I_{n} : (y_{n}) = (z_{n - 1}, z_{n}, z_{n + 1}) + K_{n}$\,$;$  
	\item[$(3)$] $J_{n} + (z_{n + 1}) = (z_{n + 1}) + L_{n}$\,$;$ 
	\item[$(4)$] $J_{n} : (z_{n + 1}) = (xz_{1}z_{2} \cdots z_{n}) + (y_{1}, \ldots, y_{n - 1})$\,$;$ 
	\item[$(5)$] $K_{n} + (y_{n - 1}) = (y_{n - 1}) + \sum_{i = 1}^{n - 2} (y_{1}, \ldots, y_{i}) (z_{i})$ \, if \, $n \ge 3$\,$;$ 
	\item[$(6)$] $K_{n} : (y_{n - 1}) = (z_{n - 2}) + K_{n - 1}$ \, if \, $n \ge 3$\,$;$  
	\item[$(7)$] $L_{n}$ 
	can be regarded as the edge ideal {\rm (\cite[p.~156]{HH})} of the Ferrers graph 
	associated with the partition 
	$(n-1, n-1, n-2, \ldots, 2, 1)$ and
	$\sum_{i = 1}^{n - 2} (y_{1}, \ldots, y_{i}) (z_{i})$ appearing in $($5$)$
	can be regarded as that with the partition
	$(n-2, n-1, \ldots, 2, 1))$, see \cite{CN}.  
	By using \cite[Theorem 2.1]{CN}, one has 
	\[
	\reg (L_{n}) = \reg \left(\sum_{i = 1}^{n - 2} (y_{1}, \ldots, y_{i}) (z_{i})\right) = 2,
	\] 
	\[
	H_{K[y_{1}, \ldots, y_{n-1}, z_{1}, \ldots, z_{n}]/L_{n}} (\lambda) = \frac{1 + (n-1)\lambda - (n-1)\lambda^{2}}{(1-\lambda)^{n}}, 
	\]
	\[
	H_{K[y_{1}, \ldots, y_{n-2}, z_{1}, \ldots, z_{n-2}] / \sum_{i = 1}^{n - 2} (y_{1}, \ldots, y_{i}) (z_{i})} (\lambda) = \frac{1 + (n-2)\lambda}{(1 - \lambda)^{n - 2}}. 
	\]
\end{enumerate}
\end{Remark}

\medskip

\begin{Lemma}\label{J}
One has 
\begin{enumerate}
	\item[$(1)$] $\reg (J_{n}) = n + 2$\,$;$
	\item[$(2)$] $\displaystyle H_{K[x, y_{1}, \ldots, y_{n-1}, z_{1}, \ldots, z_{n + 1}]/J_{n}} (\lambda) = \frac{1 + n\lambda- (n-2)\lambda^{2} + \lambda^{3} + \cdots + \lambda^{n+1}}{(1-\lambda)^{n+1}}$. 
\end{enumerate}
\end{Lemma}

\newpage

\begin{proof}
(1) Since $xz_{1}z_{2} \cdots z_{n+1}$ belongs to the unique minimal system of monomial generators of $J_{n}$, it follows that $n + 2 \le \reg (J_{n})$.  We claim $\reg (J_{n}) \le n + 2$.  By using Lemma \ref{LemB} (2) together with Remark \ref{IJKL} (3) and (7), one has 
\[
\reg (J_{n} + (z_{n + 1})) = \reg((z_{n + 1}) + L_{n})) = \reg (L_{n}) = 2.  
\]
Furthermore, Lemma \ref{LemB} (2) together with Remark \ref{IJKL} (4) says that $\reg (J_{n} : (z_{n + 1})) = n + 1$. 
Hence, $\reg (J_{n}) \le n + 2$ follows from Lemma \ref{LemC}. 

\medskip

(2) Let $T = K[x, y_{1}, \ldots, y_{n-1}, z_{1}, \ldots, z_{n + 1}]$ and consider the short exact sequence
\[
0 \to \frac{T}{J_{n} : (z_{n + 1})} (-1) \xrightarrow{\ \times z_{n + 1}\ } \frac{T}{J_{n}} \to \frac{T}{J_{n} + (z_{n + 1})} \to 0. 
\]
Remark \ref{IJKL} (3) and (7) yield
\begin{eqnarray*}
H_{T / J_{n} + (z_{n + 1})} (\lambda) &=& H_{T / (z_{n + 1}) + L_{n}} (\lambda) \\
& & \\
&=& H_{K[y_{1}, \ldots, y_{n-1}, z_{1}, \ldots, z_{n}] / L_{n}} (\lambda) \cdot \frac{1}{(1-\lambda)}\\
& & \\
&=& \frac{1 + (n-1)\lambda - (n-1)\lambda^{2}}{(1-\lambda)^{n+1}}. 
\end{eqnarray*}
Furthermore, Remark \ref{IJKL} (4) yeilds 
\begin{eqnarray*}
H_{T / J_{n} : (z_{n + 1})} (\lambda) &=& H_{T / (xz_{1}z_{2} \cdots z_{n}) + (y_{1}, \ldots, y_{n - 1})} (\lambda) \\
& & \\
&=& H_{K[x, z_{1}, \ldots, z_{n}] / (xz_{1}z_{2} \cdots z_{n})} (\lambda) \cdot \frac{1}{(1-\lambda)}\\ 
& & \\
&=& \frac{1 + \lambda + \cdots + \lambda^{n}}{(1 - \lambda)^{n+1}}. 
\end{eqnarray*}
Hence 
\begin{eqnarray*}
H_{T / J_{n}} (\lambda) &=& H_{T / J_{n} + (z_{n + 1})} (\lambda) + \lambda \cdot H_{T / J_{n} : (z_{n + 1})} (\lambda) \\
& & \\
&=& \frac{1 + n\lambda - (n-2)\lambda^{2} + \lambda^{3} + \cdots + \lambda^{n+1}}{(1-\lambda)^{n+1}}, 
\end{eqnarray*}
as desired.
\qed

\end{proof}

\begin{Lemma}
\label{K}
One has
\begin{enumerate}
	\item[$(1)$] $\reg (K_{n}) = n; $  
	\item[$(2)$] $\displaystyle H_{K[x, y_{1}, \ldots, y_{n-1}, z_{1}, \ldots, z_{n - 2}] / K_{n}} (\lambda) = \frac{1 + \sum_{i = 1}^{n - 1}(n - i)\lambda^{i}}{(1-\lambda)^{n-1}}$. 
\end{enumerate}
\end{Lemma}

\newpage

\begin{proof} 
Since $xy_{1}y_{2} \cdots y_{n-1}$ belongs to the unique minimal system of monomial generators of $J_{n}$, one has $n \le \reg (K_{n})$.  We claim $\reg (K_{n}) \le n$ and (2) by using induction on $n$.  Since $K_{2} = (xy_{1})$, each of the assertion is trivial for $n = 2$. 
Let $n > 2$. 
Lemma \ref{LemB} (2) together with Remark \ref{IJKL} (5) and (7) guarantees that
\[ 
\reg (K_{n} + (y_{n-1})) = \reg \left( (y_{n-1}) +  \sum_{i = 1}^{n - 2} (y_{1}, \ldots, y_{i}) (z_{i}) \right) = 2.
\] 
Moreover, by virtue of Lemma \ref{LemB} (2), Remark \ref{IJKL} (6) as well as the induction hypothesis, it follows that 
\[
\reg (K_{n} : (y_{n - 1})) = \reg ((z_{n-2}) + K_{n-1}) = \reg (K_{n-1}) = n - 1.
\] 
Hence, Lemma \ref{LemC} says that $\reg (K_{n}) \le n$. 

Now, consider the short exact sequence
\[
0 \to \frac{T'}{K_{n} : (y_{n - 1})} (-1) \xrightarrow{\ \times y_{n - 1}\ } \frac{T'}{K_{n}} \to \frac{T'}{K_{n} + (y_{n - 1})} \to 0, 
\]
where $T' = K[x, y_{1}, \ldots, y_{n-1}, z_{1}, \ldots, z_{n - 2}]$. 
It follows from Remark \ref{IJKL} (5) and (7) that  
\begin{eqnarray*}
H_{T' / K_{n} + (y_{n - 1})} (\lambda) &=& H_{T' / (y_{n - 1}) + \sum_{i = 1}^{n - 2} (y_{1}, \ldots, y_{i}) (z_{i})} (\lambda) \\
& & \\
&=& H_{K[y_{1}, \ldots, y_{n-2}, z_{1}, \ldots, z_{n-2}] / \sum_{i = 1}^{n - 2} (y_{1}, \ldots, y_{i}) (z_{i})} (\lambda) \cdot \frac{1}{1 - \lambda}\\
& & \\ 
&=& \frac{1 + (n-2)\lambda}{(1-\lambda)^{n-1}}. 
\end{eqnarray*}
Furthermore, Remark \ref{IJKL} (6) as well as the induction hypothesis guarantees that 
\begin{eqnarray*}
H_{T' / K_{n} : (y_{n - 1})} (\lambda) &=& K_{T' / (z_{n - 2}) + K_{n - 1}} (\lambda) \\
& & \\
&=& H_{K[x, y_{1}, \ldots, y_{n-2}, z_{1}, \ldots, z_{n - 3}] / K_{n - 1}} (\lambda) \cdot \frac{1}{1 - \lambda}\\
& & \\ 
&=& \frac{1 + \sum_{i = 1}^{n - 2}(n - 1 - i)\lambda^{i}}{(1-\lambda)^{n-1}}. 
\end{eqnarray*}

\newpage

\noindent
Hence, one has
\begin{eqnarray*}
H_{T' / K_{n}} (\lambda) &=& H_{T' / K_{n} + (y_{n - 1})} (\lambda) + \lambda \cdot H_{T' / K_{n} : (y_{n - 1})} (\lambda) \\
& & \\
&=& \frac{1 + (n-2)\lambda + \lambda + \lambda \cdot \sum_{i = 1}^{n - 2}(n - 1 - i)\lambda^{i}}{(1-\lambda)^{n-1}} \\
& & \\
&=& \frac{1 + (n-1)\lambda + \sum_{i = 1}^{n - 2}(n - 1 - i)\lambda^{i + 1}}{(1-\lambda)^{n-1}} \\
& & \\
&=& \frac{1 + (n-1)\lambda + \sum_{i = 2}^{n - 1}(n - i)\lambda^{i}}{(1-\lambda)^{n-1}} \\ 
& & \\
&=& \frac{1 + \sum_{i = 1}^{n - 1}(n - i)\lambda^{i}}{(1-\lambda)^{n-1}}, 
\end{eqnarray*}
as desired.
\qed
 
\end{proof}

The monomial ideal $I_{n}$ plays an important role in our proof of Theorem \ref{main}. 

\begin{Proposition}\label{I}
One has
\begin{enumerate}
	\item[$(1)$] $\reg (S_{n} / I_{n}) = n + 1$\,$;$ 	
	\item[$(2)$] $\displaystyle H_{S_{n} / I_{n}} (\lambda) = \frac{1 + (n + 1)\lambda}{(1-\lambda)^{n + 1}}$. 
\end{enumerate}
Thus, in particular, $\deg h_{S_{n} / I_{n}} (\lambda)= 1$. 
\end{Proposition}

\begin{proof}
(1) By virtue of Lemma \ref{LemA}, it is sufficient to show that $\reg (I_{n}) = n + 2$.  Since $xz_{1}z_{2} \cdots z_{n+1}$ belongs to the unique minimal system of monomial generators of $I_{n}$, one has $n + 2 \le \reg (I_{n})$.  We claim $\reg (I_{n}) \le n + 2$.  It follows from Lemma \ref{LemB} (2), Remark \ref{IJKL} (1) together with Lemma \ref{J} (1) that 
\[
\reg (I_{n} + (y_{n})) = \reg((y_{n}) + J_{n}) = \reg (J_{n}) = n + 2.
\] 
By using Lemma \ref{LemB} (2), Remark \ref{IJKL} (2) together with Lemma \ref{K} (1), one has
\[
\reg (I_{n} : (y_{n})) = \reg((z_{n - 1}, z_{n}, z_{n + 1}) + K_{n}) = \reg (K_{n}) = n.
\] 
Hence Lemma \ref{LemC} says that $\reg (I_{n}) \le n + 2$, as desired.

\medskip

(2) Considering the short exact sequence

\[
0 \to \frac{S_{n}}{I_{n} : (y_{n})} (-1) \xrightarrow{\ \times y_{n}\ } \frac{S_{n}}{I_{n}} \to \frac{S_{n}}{I_{n} + (y_{n})} \to 0.  
\]
Remark \ref{IJKL} (1) together with Lemma \ref{J} (2) yields 
\begin{eqnarray*}
H_{S_{n} / I_{n} + (y_{n})} (\lambda) &=& H_{S_{n} / (y_{n}) + J_{n}} (\lambda) \\
&=& H_{K[x, y_{1}, \ldots, y_{n-1}, z_{1}, \ldots, z_{n+1}] / J_{n}} (\lambda) \\
& & \\
&=& \frac{1 + n\lambda - (n-2)\lambda^{2} + \lambda^{3} + \cdots + \lambda^{n+1}}{(1-\lambda)^{n+1}}. 
\end{eqnarray*}
Furthermore, Remark \ref{IJKL} (2) together with Lemma \ref{K} (2) yields  
\begin{eqnarray*}
H_{S_{n} / I_{n} : (y_{n})} (\lambda) &=& H_{S_{n} / (z_{n-1}, z_{n}, z_{n+1}) + K_{n}} (\lambda) \\
& & \\
&=& H_{K[x, y_{1}, \ldots, y_{n-1}, z_{1}, \ldots, z_{n-2}] / K_{n}} (\lambda) \cdot \frac{1}{1-\lambda}\\
& & \\
&=& \frac{1 + \sum_{i = 1}^{n - 1}(n - i)\lambda^{i}}{(1-\lambda)^{n}}. 
\end{eqnarray*}
It then follows that
\begin{eqnarray*}
& & H_{S_{n} / I_{n}} (\lambda) = H_{S_{n} / I_{n} + (y_{n})} (\lambda) + \lambda \cdot H_{S_{n} / I_{n} : (y_{n})} (\lambda) \\
& & \\
&=& \frac{1 + n\lambda - (n-2)\lambda^{2} + \lambda^{3} + \cdots + \lambda^{n+1}}{(1-\lambda)^{n+1}} + \frac{t + \sum_{i = 1}^{n - 1}(n - i)\lambda^{i+1}}{(1-\lambda)^{n}} \\
& & \\
&=& \frac{1 + n\lambda - (n-2)\lambda^{2} + \lambda^{3} + \cdots + \lambda^{n+1} + \lambda(1-\lambda) + (1-\lambda)\sum_{i = 1}^{n - 1}(n - i)\lambda^{i+1}}{(1-\lambda)^{n+1}} \\
& & \\
&=& \frac{1 + (n + 1)\lambda - (n-1)\lambda^{2} + \sum_{i = 2}^{n} \lambda^{i + 1}+ \sum_{i = 1}^{n - 1}(n - i)\lambda^{i+1} - \sum_{i = 1}^{n - 1}(n - i)\lambda^{i+2}}{(1-\lambda)^{n+1}} \\
& & \\
&=& \frac{1 + (n + 1)\lambda - (n-1)\lambda^{2} + \sum_{i = 2}^{n} \lambda^{i + 1}+ \sum_{i = 1}^{n - 1}(n - i)\lambda^{i+1} - \sum_{i = 2}^{n}(n - i + 1)\lambda^{i+1}}{(1-\lambda)^{n+1}} \\
& & \\
&=& \frac{1 + (n + 1)\lambda - (n-1)\lambda^{2} + \sum_{i = 2}^{n} \lambda^{i + 1} + (n-1)\lambda^{2} - \sum_{i = 2}^{n} \lambda^{i + 1}}{(1-\lambda)^{n+1}} \\
& & \\
&=& \frac{1 + (n + 1)\lambda}{(1-\lambda)^{n + 1}},
\end{eqnarray*}
as required.
\qed

\end{proof}

We are now in the position to finish a proof of Theorem \ref{main}.

\newpage

\begin{proof}({\em Proof of Theorem \ref{main}.})
Let $r$ and $s$ be positive integers with $r, s \geq 1$. 
By virtue of Proposition \ref{PropA}, only the case of $r > s \geq 1$ will be discussed.
%
Let $S_{1} = K[x, y_{1}, z_{1}, z_{2}]$ and $I_{1} = (xy_{1}, xz_{1}z_{2}, y_{1}z_{1}, y_{1}z_{2})$.  It then follows that $\reg (S_{1} / I_{1}) = 2$ and $H_{S_{1} / I_{1}} (\lambda) = \frac{1 + 2\lambda}{(1-\lambda)^{2}}$. 
By virtue of this example and of Proposition \ref{I}, one has
\[
\reg (S_{r-s} / I_{r-s}) = r - s + 1, \ \ H_{S_{r-s} / I_{r-s}} (\lambda) = \frac{1 + (r-s + 1)\lambda}{(1-\lambda)^{r-s + 1}}. 
\]
Let $S = S_{r-s} \otimes_{K} K[u_{1}, \ldots, u_{s}]$ and $I = I_{r-s} + (u_{1}u_{2} \cdots u_{s})$.  Lemma \ref{LemB} (3) together with Proposition \ref{PropA} yields 
\[
\reg (S / I) = r - s + 1 + s - 1 = r
\]
and
\[
H_{S / I} (\lambda) = \frac{ \{1 + (r-s+1)\lambda\}(1 + \lambda + \cdots + \lambda^{s-1})}{(1-\lambda)^{r}}. 
\]
Hence $\deg h_{S / I} (\lambda) = s$ and $I$ is a desired monomial ideal.
\qed
 
\end{proof}

\bigskip

\section{Examples}

The purpose of this section is to give a class of edge ideals $I \subset S$ of {\em Cameron--Walker graphs} (\cite{HHKO}) with $\reg(S/I) = \deg h_{S/I}(\lambda)$ for which $S/I$ is not Cohen--Macaulay. 

Let $G$ be a finite simple graph on the vertex set $[n] = \{1, \ldots, n\}$ and $E(G)$ its edge set.  (A finite graph is called {\em simple} if it possesses no loop and no multiple edge.)  The {\em edge ideal} $I(G)$ of $G$ is the monomial ideal of $S = K[x_1, \ldots. x_n]$ generated by those quadratic monomials $x_{i}x_{j}$ with $\{i, j\} \in E(G)$: 
\[
I(G) = ( \, x_{i}x_{j} \, : \,  \{i, j\} \in E(G) \, ) \subset S. 
\] 
In general, it is quite difficult to compute the regularity of an edge ideal.  However, one can compute $\reg (I(G))$ easily if $G$ is a Cameron--Walker graph.  The notion of Cameron-Walker graph was introduced by \cite{CW}.  We refer the reader to \cite[p.~258]{HHKO} for a classification of Cameron--Walker graphs. 

\begin{Example}
{\em
Fix $m \ge 1$ and write $G^{1}_{m}$ for the star triangle joining $m$ triangles 
at one common vertex. Then $S/I(G^{1}_{m})$ is Cohen--Macaulay if and only if $m = 1$ (\cite[Theorem 1.3]{HHKO}).  Hence
\begin{itemize}
	\item If $m = 2k$, then $\reg (S/I(G^{1}_{m})) = 2k > 2k - 1 = \deg h_{S/I(G^{1}_{m})} (\lambda)$.
	\item If $m = 2k + 1$, then $\reg (S/I(G^{1}_{m})) = \deg h_{S/I(G^{1}_{m})} (\lambda) = 2k + 1$. 
\end{itemize}
}
\end{Example}

\newpage

\begin{Example}
{\em 
Fix $m \ge 1$ and write $G^{2}_{m}$ for the graph drawn below on the vertex set $[2m + 3]$:

\bigskip
 
\begin{xy}
	\ar@{} (0,0);(40,-8) *{\text{$G^2_{m} =$}};
	\ar@{} (0,0);(70, 8) *++!L{2m+3} *\cir<4pt>{} = "A";
	\ar@{-} "A";(70, 0) *++!L{2m+2} *\cir<4pt>{} = "B";
	\ar@{-} "B";(70, -8) *++!L{2m+1} *\cir<4pt>{} = "C";
	\ar@{-} "C";(50, -20) *++!R{1} *\cir<4pt>{} = "D";
	\ar@{-} "C";(54, -24) *++!U{2} *\cir<4pt>{} = "E";
	\ar@{-} "D";"E";
	\ar@{} "A"; (70, -13) *++!U{\cdots}
	\ar@{-} "C";(90, -20) *++!L{2m} *\cir<4pt>{} = "F";
	\ar@{-} "C";(86, -24) *++!U{2m-1} *\cir<4pt>{} = "G";
	\ar@{-} "F";"G"; 
\end{xy}

\bigskip

\noindent
Then \cite[Theorem 1.3]{HHKO} says that $G^{2}_{m}$ is not Cohen--Macaulay if $m \ge 2$.  However,
\[
\reg (S/I(G^{2}_{m})) = \deg h_{S/I(G^{2}_{m})} (\lambda) = m + 1.
\] 
On the other hand, this graph is of interest from the viewpoint of $h$-vector 
$(h_{0}, h_{1}, \ldots, h_{s})$, which is the sequence of coefficients of $h$-polynomial.  It is known \cite[Theorem 4.4]{S} that the $h$-vector of Gorenstein ring is symmetric, but the converse is false.  A routine computation shows that 
\[
h_{S/I(G^{2}_{m})} (\lambda) =  (1 + \lambda)^{m + 1} + \lambda(1-\lambda)^{m-1}. 
\]
Hence $h$-vector of $S/I(G^{2}_{m})$ is symmetric if $m$ is odd, but not necessary unimodal.  For example, the $h$-vector of $S/I(G^{2}_{3})$ is $(1,5,4,5,1)$, which is not unimodal.  In general, the $h$-vector of $S/I(G^{2}_{m})$ is not unimodal if $m = 4k + 3$.
} 
\end{Example}

\bigskip

\noindent
{\bf Acknowledgment.}
The first author was partially supported by JSPS KAKENHI 26220701.
The second author was partially supported by JSPS KAKENHI 17K14165.

\end{document}